\documentclass[12pt]{article}
\usepackage{amsmath,amssymb,amsthm}
\usepackage{graphicx}
  \newtheorem{theorem}{Theorem}

 \newtheorem{lemma}[theorem]{Lemma}

 {\theoremstyle{definition}
 \newtheorem{definition}[theorem]{Definition}
 \newtheorem{remark}[theorem]{Remark}

 }
\usepackage{buczo}
\renewcommand{\ggg}{\gamma}

\renewcommand{\lll}{\lambda}
\title{Multifractal properties of convex hulls of typical continuous functions}
\author{Zolt\'an Buczolich\thanks{
Research supported by by the Hungarian
National Foundation for Scientific Research Grant K104178.
},
Department of Analysis, E\"otv\"os Lor\'and\\
University, P\'azm\'any P\'eter S\'et\'any 1/c, 1117 Budapest, Hungary\\
email: buczo@cs.elte.hu\\
{\tt www.cs.elte.hu/\hbox{$\sim$}buczo}
}
\date{\today}
\begin{document}
\maketitle

\medskip


\begin{abstract}
We study the singularity (multifractal)  spectrum of the convex hull of the typical/generic continuous functions  defined on $[0,1]^{d}$.  We denote by  ${\mathbf E}_\fff^{h} $ the set of points at which $\fff: {[0,1]^d}\to {\ensuremath {\mathbb R}}$ has a pointwise 
H\"older exponent equal to $h$.
Let $H_{f}$ be the convex hull of the graph of $f$,
the concave function on the top of $H_{f}$ is denoted by
$\fff_{1,f}(\bbx)=\max \{y:(\bbx,y)\in H_{f} \}$  and
$\fff_{2,f}(\bbx)=\min \{y:(\bbx,y)\in H_{f} \}$ denotes the convex function on the bottom of $H_{f}$.
 We show that there is a dense $G_\ddd$ subset $\cag\sse {C[0,1]^d}$ such that 
for $f\in \cag$ the following properties are satisfied.
For $i=1,2$ the functions $ {\fff_ {i,f}}$ and $f$ coincide only on a set of zero Hausdorff dimension, the functions $ {\fff_ {i,f}}$ are continuously differentiable
on $(0,1)^{d}$,
${\mathbf E}_{\fff_{i,f}}^{0} $ equals the boundary of $ {[0,1]^d}$,
$\dim_{H}{\mathbf E}_{\fff_{i,f}}^{1}=d-1 $, $\dim_{H}{\mathbf E}_{\fff_{i,f}}^{+\oo}=d $ and ${\mathbf E}_{\fff_{i,f}}^{h}=\ess$ if $h\in(0,+\oo)\sm \{1 \}$.
\end{abstract}

\begin{figure}
  \begin{center}
  \includegraphics[width=14.1cm,height=8.1cm]{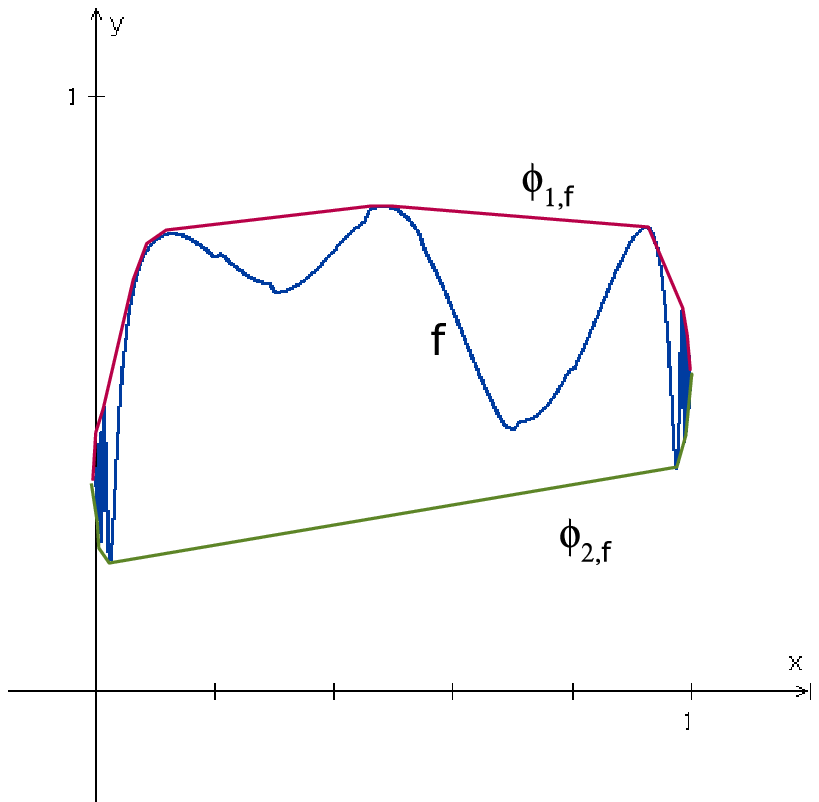}
    \caption{$f$, $ {\fff_ {1,f}}$ and $ {\fff_ {2,f}}$ in 1D}
     \label{fig1}
  \end{center}
\end{figure}

\section{Introduction}

We started with J. Nagy our study of multifractal properties of typical/generic functions
in \cite{BUC} where multifractal properties of generic monotone
functions on $[0,1]$ were treated.
The higher dimensional version of this question was considered
in \cite{BSJMAA} where  with S. Seuret  we investigated the H\"older spectrum of   functions monotone  in several variables. 
In  \cite{BuS2} we also showed that typical Borel measures on $[0,1]^d$ satisfy a multifractal formalism.
Multifractal properties of typical convex continuous functions defined on $ {[0,1]^d}$
are discussed in \cite{BStcv}.

In \cite{[BruHa]} the convex hull of typical continuous functions $f\in C[0,1]$ is considered by A. M. Bruckner and J. Haussermann. In this case the boundary of this convex hull
decomposes into two functions (in our notation) $ {\fff_ {1,f}}$ and $ {\fff_ {2,f}}$ see Figure \ref{fig1}.
The upper one, $ {\fff_ {1,f}}$ is concave the lower one $ {\fff_ {2,f}}$ is convex.
It is shown that for the typical $f$ these functions are continuously
differentiable on $(0,1)$ and at the endpoints they have infinite derivatives.
The aim of our paper is to describe the multifractal spectrum of these functions
in the multidimensional setting. That is, 
generic/typical continuous functions $f$ in $C {[0,1]^d}$ are considered in the sense of Baire category. The topology on $C {[0,1]^d}$ is
the supremum metric. 
 We also prove the multidimensional version 
of the above mentioned  results of A. M. Bruckner and J. Haussermann.

The convex hull of the graph of $f\in  {C[0,1]^d}$ is denoted by $H_{f}$,
that is, $H_{f}=\text{  convex hull of  }\{(\bbx,f(\bbx)):\bbx\in {[0,1]^d}  \}.$

We are interested in two functions: $\fff_{1,f}(\bbx)=\max \{y:(\bbx,y)\in H_{f} \}$ is the function on the top of $H_{f}$, and
$\fff_{2,f}(\bbx)=\min \{y:(\bbx,y)\in H_{f} \}$ is the function on the bottom of $H_{f}$. On Figure \ref{fig1} these functions are illustrated in 
dimension one.

The points where $f$ and $\fff_{i,f}$ coincide are denoted by $E_{i,f}=
\{\bbx: \fff_{i,f}(\bbx)=f(\bbx)\}$, $i=1,2$.

The H\"older exponent and singularity spectrum  for a locally bounded function is defined as follows.
 
\begin{definition}
Let $f \in L^\infty( {[0,1]^ {d}})$. For $h\geq 0$ and ${\mathbf {x}}\in  {[0,1]^ {d}}$,
the function $f$  belongs to $C^h_{{\mathbf {x}}}$ if there are a polynomial $P$ of degree less than $[h]$ and a constant $C$ such that,  for ${\mathbf {x}}'$ close to  ${\mathbf {x}}$,
\begin{equation}
 \label{defpoint}
|f({\mathbf {x}}') - P({\mathbf {x}}'-{\mathbf {x}})| \leq C |{\mathbf {x}}' -{\mathbf {x}}|^h.
\end{equation}
The pointwise  {H\"older} exponent of $f$ at ${\mathbf {x}}$ is $h_f({\mathbf {x}}) =  \sup\{h\geq 0: \ f\in C^h_{{\mathbf {x}}} \}.$ \end{definition}

\begin{definition}The
 singularity spectrum of  $f$  is defined by 
 $$d_f(h)=\dim_{H}  {\mathbf E}_f^{h} , \ \ \text{  where  }  {\mathbf E}_f^{h} =\{{\mathbf {x}}: h_f({\mathbf {x}})=h\}.$$
  
Here $\dim_{H}$ denotes the Hausdorff dimension, and $\dim \emptyset = -\infty$ by convention.

 We will also use the sets
 \begin{equation}
  \label{defep}
  {\mathbf E}_f^{h,\leq} =\{{\mathbf {x}}: h_{f}({\mathbf {x}})\leq h  \} \supset  {\mathbf E}_f^{h} \text{ and }{\mathbf E}_f^{h,<} =\{{\mathbf {x}}: h_{f}({\mathbf {x}})< h  \} .
  \end{equation}
  \end{definition}

The faces of $ {[0,1]^d}$ are $$F_{0,j}=\{({x_{1},...,x_{j-1},0,x_{j+1},...,x_{d}})\in {[0,1]^d} \}$$
and
$$F_{1,j}=\{({x_{1},...,x_{j-1},1,x_{j+1},...,x_{d}})\in {[0,1]^d} \}.$$
 Then $\dd( {[0,1]^d})=\cup_{i=1}^{2}\cup_{j=1}^{d}F_{i,j}$.

  The main result of our paper is the following theorem:
\begin{theorem}\label{*30th}
There exists a dense $G_{\ddd}$ set $\cag\sse {C[0,1]^d}$ such that for every $f\in\cag$
for $i=1,2$
\begin{itemize}
\item $ {\fff_ {i,f}}$ is continuously differentiable on $ {(0,1)^d}$;
\item if $\bbx\in\dd( {[0,1]^d})$ then
\begin{equation}\label{*30c}
h_{{\fff_ {i,f}}}(\bbx)=0
\end{equation}
hence $d_{{\fff_ {i,f}}}(0)=d-1$ and ${\mathbf E}_{\fff_{i,f}}^{0}\cap(0,1)^{d}=\ess $;
\item $d_{{\fff_ {i,f}}}(1)=d-1$;
\item $d_{{\fff_ {i,f}}}(+\oo)=d$;
\item $d_{{\fff_ {i,f}}}(h)=-\oo$, that is, ${\mathbf E}_f^{h}=\ess$ for $h\in(0,+\oo)\sm \{1 \}$;
\item for $j=1,...,d$ if $\bbx\in F_{0,j}$ then
\begin{equation}\label{*30b}
\dd_{j,+} {\fff_ {i,f}}(\bbx)=(-1)^{i+1}(+\oo)
\end{equation}
if $\bbx\in F_{1,j}$ then
\begin{equation}\label{*30a}
\dd_{j,-} {\fff_ {i,f}}(\bbx)=(-1)^{i}(+\oo);
\end{equation}
\item $\dim_{H}E_{i,f}=0$.
\end{itemize}
\end{theorem}

\section{Notation and preliminary results}

The open ball with center $\bbx$ and of radius $r>0$ is denoted by $B(\bbx,r)$. We use similar notation
for open neighborhoods of subsets, for example if $A\sse  {\ensuremath {\mathbb R}}^{d}$ then
$B(A,r)=\{\bbx\in  {\ensuremath {\mathbb R}}^{d}: \dist(\bbx,A)<r \}.$

For subsets $A\sse  {\ensuremath {\mathbb R}}^{d}$ we denote the diameter of the set $A$
by $|A|$ while $\dd A$ denotes is boundary.

The $j$'th basis vector in $ {\ensuremath {\mathbb R}}^{d}$ is denoted by $\bbe _{j}=(0,...,0,\underset{\underset{j}{\uparrow}}{1},0,...,0).$

In our proofs we use a standard fixed countable dense set 
$\{f_{n} \}_{n=1}^{\oo}$ in $ {C[0,1]^d}$.
We assume that the functions $f_{n}$ are in $C^{\oo} {[0,1]^d}$.
(Taking an arbitrary countable dense set $\{\tf_{n}\}$ in $C {[0,1]^d}$ by using a mollifier function it is easy to obtain such a dense set of $C^{\oo}$ functions.)

\begin{definition}\label{*1bbdef}
We say that $f: {[0,1]^d}\to  {\ensuremath {\mathbb R}}$ is {\it piecewise linear}
 if there is a partition $Z_{j}$, $j=1,...,n_{\zzz}$
 of $ {[0,1]^d}$ into simplices such that for each $j$ the set
 $\{(\bbx,f(\bbx)):\bbx \in Z_{j} \}$ is the subset of a hyperplane
 in $ {\ensuremath {\mathbb R}}^{d+1}$.\\
 We say that $f$ is {\it independent piecewise linear} if it is piecewise
 linear and if $V$ denotes the collection of the vertices
 of the simplices $Z_{j}$, $j=1,...,n_{\zzz}$ for a suitable
 partition then 
 \begin{itemize}
 \item from
 $\bbx_{1},...,\bbx_{k}\in V$, $\bbx_{i}\not=\bbx_{j}$, $i\not=j$,
 the points 
$(\bbx_{j},f(\bbx_{j}))$, $j=1,...,k$ are on the same $d$-dimensional hyperplane in $ {\ensuremath {\mathbb R}}^{d+1}$ it
follows that $k\leq d+1$,
\item from
 $\bbx_{1},...,\bbx_{k}\in V\cap  {(0,1)^d}$, $\bbx_{i}\not=\bbx_{j}$, $i\not=j$,
 the points 
$\bbx_{j}$, $j=1,...,k$ are on the same $(d-1)$-dimensional hyperplane in $ {\ensuremath {\mathbb R}}^{d}$ it
follows that $k\leq d$.
\end{itemize}
 \end{definition}

The second assumption in the above defintion is void in the one dimensional case since the zero dimensional
hyperplanes are just points.

 \begin{lemma}\label{*1clem}
 Suppose $f\in  {C[0,1]^d}$.
 For any $ {{\mathbf x}_0}\in {[0,1]^d}$ there exist $ {{\mathbf x}_i}\in E_{1,f},$ $i=1,...,d+1$
 and $p_{i}\geq 0$, ${\sum_{i=1}^{d+1}} p_{i}=1$, such that $ {{\mathbf x}_0}=\sum_{i=1}^{d+1} p_{i}{{\mathbf x}_i}$ and $$ {\fff_ {1,f}}( {{\mathbf x}_0})=\sum_{i=1}^{d+1} p_{i} {\fff_ {1,f}}( {{\mathbf x}_i})=
 \sum_{i=1}^{d+1} p_{i}f( {{\mathbf x}_i}).$$
 \end{lemma}
 
\begin{remark}
We remark that in the above lemma some $p_{i}$'s can equal zero,
or some $ {{\mathbf x}_i}$'s coincide.
\end{remark}
 
 \begin{proof}
 By Carath\'eodory's theorem (see for example \cite{Eck}) from $( {{\mathbf x}_0}, {\fff_ {1,f}}( {{\mathbf x}_0}))\in H_{f}$ it follows that one can find (not necessarily different) $ {{\mathbf x}_i}\in {[0,1]^d}$, $p_{i}\geq 0$, $i=1,...,d+1$, ${\sum_{i=1}^{d+1}} p_{i}=1$
  such that ${\sum_{i=1}^{d+1}} p_{i} {{\mathbf x}_i}= {{\mathbf x}_0}$, ${\sum_{i=1}^{d+1}} p_{i} f( {{\mathbf x}_i})=
   {\fff_ {1,f}}( {{\mathbf x}_0}).$ If for an $i'$ we had $f( {{\mathbf x}_ {i'}})< {\fff_ {1,f}}( {{\mathbf x}_ {i'}})$
  then letting $$\ds y'=\Big (\sum_{i=1,\ i\not = i'}^{d+1}p_{i}f( {{\mathbf x}_i})\Big )+p_{i'} {\fff_ {1,f}}( {{\mathbf x}_ {i'}})$$ we would obtain $( {{\mathbf x}_0},y')\in H_{f}$
  contradicting that $ {\fff_ {1,f}}( {{\mathbf x}_0})=\max\{y:( {{\mathbf x}_0},y)\in H_{f} \}$.
 \end{proof}
 
 \begin{lemma}\label{*2lem}
 There exists a dense $G_{\ddd}$ set $\cag_{0,d}\sse  {C[0,1]^d}$  such that 
for every $f\in\cag_{0,d}$ and for every $ {{\mathbf x}_0}\in F_{0,d}$
\begin{equation}\label{*2a}
\dd_{d,+} {\fff_ {1,f}}( {{\mathbf x}_0})=+\oo,  \dd_{d,+} {\fff_ {2,f}}( {{\mathbf x}_0})=-\oo \text{  and  }
\end{equation} 
\begin{equation}\label{*2b}
h_{\fff_{i},f}( {{\mathbf x}_0})=0 \text{  for  }i=1,2.
\end{equation}
 \end{lemma}
 
 \begin{proof}
 Without limiting generality we prove the statement for $ {\fff_ {1,f}}$.
 \\
 Suppose $ {{\mathbf x}_0}=(x_{0,1},...,x_{0,d-1},0)\in F_{0,d}$.
 By Lemma \ref{*1clem}
 there exist $ {{\mathbf x}_i}\in E_{1,f}\sse  {[0,1]^d}$, and $p_{i}\geq 0$, $i=1,...,d+1$,
  such that  ${\sum_{i=1}^{d+1}} p_{i}=1$,  
  \begin{equation}\label{*3a}
  {\sum_{i=1}^{d+1}} p_{i}\bbx_{i}=\bbx_{0} \text{  and  }
  {\sum_{i=1}^{d+1}} p_{i}f( {{\mathbf x}_i})={\sum_{i=1}^{d+1}} p_{i} {\fff_ {1,f}}( {{\mathbf x}_i})= {\fff_ {1,f}}( {{\mathbf x}_0}).
  \end{equation}
  This and $\bbx_{0}\in F_{0,d}$ imply 
  \begin{equation}\label{*3b}
   {{\mathbf x}_i}\in F_{0,d},\  i=1,...,d+1.
  \end{equation}
  Put $M_{n}=||f_{n}'||_{\oo}\geq |\dd_{d}f(\bbx)|$ for all $\bbx\in {[0,1]^d}$.
  We also let
  \begin{equation}\label{*3c}
  f_{n,m}(\bbx)=f_{n}(\bbx)+\frac{1}{n+m}(\dist(\bbx,F_{0,d}))^{1/m}
  \text{  and  } {\delta_ {n,m}}=\frac{1}{(n+m)2^{(n+m)}}.
  \end{equation}
  It is clear that $G_{m}=\cup_{n=1}^{\oo}B(f_{n,m}, {\delta_ {n,m}})$ is dense and open
  in $ {C[0,1]^d}$ and $\cag_{0,d}\defeq \cap_{m=1}^{\oo}G_{m}$ is dense $G_{\ddd}$.
  Suppose $f\in\cag_{0,d}$. Then there exists a sequence $n_{m}$, $m=1,...$
 such that $f\in B(f_{n_m,m}, {\delta_ {n_m,m}})$.
Since $ {{\mathbf x}_i}\in F_{0,d}$ by \eqref{*3c} we have $$f( {{\mathbf x}_i})\leq f_{n_m,m}( {{\mathbf x}_i})+ {\delta_ {n_m,m}}=f_{n_{m}}( {{\mathbf x}_i})+ {\delta_ {n_m,m}},\ i=1,...,d+1.$$
Therefore, using \eqref{*3a}
$$ {\fff_ {1,f}}( {{\mathbf x}_0})\leq \Big ({\sum_{i=1}^{d+1}} p_{i}f_{n_{m}}( {{\mathbf x}_i})\Big )+ {\delta_ {n_m,m}}.$$
On the other hand, since ${\fff_ {1,f}}$ is concave
$$ {\fff_ {1,f}}( {{\mathbf x}_0}+t {{\mathbf e}_d})\geq  \Big ({\sum_{i=1}^{d+1}} p_{i} {\fff_ {1,f}}( {{\mathbf x}_i}+t {{\mathbf e}_d})\Big )\geq$$
$$ {\sum_{i=1}^{d+1}} p_{i}f( {{\mathbf x}_i}+t {{\mathbf e}_d})\geq  \Big ({\sum_{i=1}^{d+1}} p_{i}f_{{n_m},m}( {{\mathbf x}_i}+t {{\mathbf e}_d})\Big )- {\delta_ {n_m,m}}\geq$$
(using \eqref{*3b} and \eqref{*3c})
$$\geq  \Big ({\sum_{i=1}^{d+1}} p_{i}f_{n_m}( {{\mathbf x}_i}+t {{\mathbf e}_d})\Big )+\frac{1}{n_{m}+m}t^{1/m}- {\delta_ {n_m,m}}.$$
Taking difference
\begin{equation}\label{*5a}
 {\fff_ {1,f}}( {{\mathbf x}_0}+t {{\mathbf e}_d})- {\fff_ {1,f}}( {{\mathbf x}_0})\geq
\end{equation}
$$ \Big ({\sum_{i=1}^{d+1}} p_{i}(f_{n_m}( {{\mathbf x}_i}+t {{\mathbf e}_d})-f_{n_m}( {{\mathbf x}_i}))\Big )-2 {\delta_ {n_m,m}}+\frac{1}{n_{m}+m}t^{1/m}\geq$$
(by the Mean value Theorem and the choice of $M_{n}$)
$$\geq \frac{1}{n_{m}+m}t^{1/m}-2 {\delta_ {n_m,m}}-M_{n}t.$$
Choosing $t_{m}=2^{-(n_{m}+m)}$ by \eqref{*3c} we obtain
$$ {\delta_ {n_m,m}}=t_{m}\frac{1}{n_{m}+m}.$$
Hence, $$\limsup_{m\to\oo}\frac{{\fff_ {1,f}}( {{\mathbf x}_0}+t_{m} {{\mathbf e}_d})- {\fff_ {1,f}}( {{\mathbf x}_0})}{t_{m}^{\aaa}}=+\oo\text{  for any  }\aaa>0.$$
This implies \eqref{*2b}.

Taking $\aaa=1$ and using concavity of $ {\fff_ {1,f}}$ we also obtain
$\dd_{d,+} {\fff_ {1,f}}( {{\mathbf x}_0})=+\oo$. This implies \eqref{*2a}.
  \end{proof}

\begin{lemma}\label{*6lem}
There exists a dense open set $\cag_{1}\sse  {C[0,1]^d}$  such that 
for every $f\in \cag_{1}$ the functions $ {\fff_ {1,f}}$ and $ {\fff_ {2,f}}$ are both
continuously differentiable on $(0,1)^{d}$.
\end{lemma}

\begin{remark}\label{*rem7}
This also implies that $h_{{\fff_ {i,f}}}(\bbx)\geq 1$
for any $\bbx\in(0,1)^{d}$, that is,
${\mathbf E}_{{\fff_ {i,f}}}^{1,<} \cap (0,1)^{d}=\ess$
for $f\in \cag_{1}$ and $i=1,2.$
\end{remark}

\begin{proof}
Again we start with $f_{n}\in C^{\oo} {[0,1]^d}$, $n=1,...$ a countable dense set
in $ {C[0,1]^d}$. This time we select $$M_{n}\geq 1 \text{  such that  }|\dd_{j}^{2}f_{n}|\leq M_{n},\ j=1,...,d.$$
We also put
\begin{equation}\label{*9a}
\ddd {_ {n,m}}= {\Big (} \frac{1}{m\cdot M_{n}}  {\Big )} ^{2}.
\end{equation} 
We put $\cag_{1}=\cap_{m}\cup_{n}B(f_{n},\ddd {_ {n,m}})$ and select
$f\in \cag_{1}.$ Then we can choose $n_{m}$  such that $f\in B( {f_ {n_m}}, {\delta_ {n_m,m}}).$

Suppose $\bbx\in  {(0,1)^d}$. We need to verify that $\dd_{j} {\fff_ {i,f}}(\bbx)$
exists and continuous for any $j=1,...,d$ and $i=1,2$.

Since the other cases are similar we can suppose that $i=1$, $j=1$.

Since $ {\fff_ {1,f}}(\bbx+t {{\mathbf e}_1})$ is a concave function in $t$ it is sufficient
to verify that its derivative exists at $t=0$ for any choice of $\bbx\in {(0,1)^d}$. This will imply that $\dd_{1} {\fff_ {1,f}}(\bbx+t {{\mathbf e}_1})$ is monotone decreasing
in $t$, without any jump discontinuities, hence  for a fixed $\bbx$ it is continuous as a function of one variable. In the end of this proof we will provide a standard
argument showing that from the concavity and continuity of ${\fff_ {1,f}}$
one can deduce that $\dd_{1} {\fff_ {1,f}}$ is continuous on $(0,1)^{d}$.

From now on $\bbx\in {(0,1)^d}$ is fixed.

By Carath\'eodory's theorem we can select $ {{\mathbf x}_i}\in E_{1,f}$, $p_{i}\geq 0$, $i=1,...,d+1$ such that ${\sum_{i=1}^{d+1}} p_{i}=1$, ${\sum_{i=1}^{d+1}} p_{i}\bbx_{i}=\bbx$ and
$$ {\fff_ {1,f}}(\bbx)={\sum_{i=1}^{d+1}} p_{i} {\fff_ {1,f}}( {{\mathbf x}_i})={\sum_{i=1}^{d+1}} p_{i}f( {{\mathbf x}_i}).$$
Suppose
\begin{equation}\label{*10a}
h_{m}=\frac{1}{m\cdot M_{n_{m}}}.
\end{equation}

By the one dimensional Taylor's formula one can find $c_{n_{m},i,\pm}$
 such that $|c_{n_{m},i,\pm}|< h_{m}$ and
 \begin{equation}\label{*11a}
  {f_ {n_m}} ( {{\mathbf x}_i}\pm h_{m} {{\mathbf e}_1})= {f_ {n_m}}( {{\mathbf x}_i})\pm \dd_{1} {f_ {n_m}}( {{\mathbf x}_i})h_{m}+
 \frac{\dd_{1}^{2} {f_ {n_m}}( {{\mathbf x}_i}+c_{n_{m},i,\pm} {{\mathbf e}_1})}{2!}h_{m}^{2}\geq
 \end{equation}
 $$ {f_ {n_m}}( {{\mathbf x}_i})\pm\dd_{1} {f_ {n_m}}( {{\mathbf x}_i})h_{m}-\frac{M_{n_{m}}}{2}h_{m}^{2}.$$
 This implies
 \begin{equation}\label{*11b}
  {\fff_ {1,f}}( {{\mathbf x}}\pm h_{m} {{\mathbf e}_1})\geq {\sum_{i=1}^{d+1}} p_{i} {\fff_ {1,f}}( {{\mathbf x}_i}\pm h_{m} {{\mathbf e}_1})\geq 
 \end{equation}
 $${\sum_{i=1}^{d+1}} p_{i}f( {{\mathbf x}_i}\pm h_{m} {{\mathbf e}_1})\geq  {\Big (} {\sum_{i=1}^{d+1}} p_{i} {f_ {n_m}}( {{\mathbf x}_i}\pm h_{m} {{\mathbf e}_1}) {\Big )} - {\delta_ {n_m,m}}\geq $$
(using \eqref{*11a})
$$ \geq  {\Big (} {\sum_{i=1}^{d+1}} p_{i} {f_ {n_m}} ( {{\mathbf x}_i}) {\Big )} \pm  {\Big (}
{\sum_{i=1}^{d+1}} p_{i}\dd_{1} {f_ {n_m}}( {{\mathbf x}_i})  {\Big )} h_{m}-\frac{M_{n_{m}}}{2}h_{m}^{2}- {\delta_ {n_m,m}}\geq \circledast$$
using $|f- {f_ {n_m}}|< {\delta_ {n_m,m}}$,  $f( {{\mathbf x}_i})= {\fff_ {1,f}}( {{\mathbf x}_i})$ and ${\sum_{i=1}^{d+1}} p_{i}f( {{\mathbf x}_i})= {\fff_ {1,f}}(\bbx)$ we can continue by
$$\circledast \geq  {\fff_ {1,f}}(\bbx)- {\delta_ {n_m,m}}\pm  {\Big (} {\sum_{i=1}^{d+1}} \dd_{1} {f_ {n_m}}( {{\mathbf x}_i})
 {\Big )} h_{m}-\frac{M_{n_{m}}}{2}h_{m}^{2}- {\delta_ {n_m,m}}.$$
By \eqref{*9a} and \eqref{*10a} we obtain that
$$\frac{{\fff_ {1,f}}(\bbx)- {\fff_ {1,f}}(\bbx-h_{m} {{\mathbf e}_1})}{h_{m}}\leq
 {\Big (} {\sum_{i=1}^{d+1}} \dd_{1} {f_ {n_m}}( {{\mathbf x}_i})
 {\Big )} +  {\Big (} \frac{M_{n_{m}}}{2} + 2  {\Big )} h_{m}
$$
and similarly
$$\frac{{\fff_ {1,f}}(\bbx+h_{m} {{\mathbf e}_1})- {\fff_ {1,f}}(\bbx)}{h_{m}}\geq
 {\Big (} {\sum_{i=1}^{d+1}} \dd_{1} {f_ {n_m}}( {{\mathbf x}_i})
 {\Big )} -  {\Big (} \frac{M_{n_{m}}}{2} + 2  {\Big )} h_{m}
$$
and hence,
$$
\frac{{\fff_ {1,f}}(\bbx)- {\fff_ {1,f}}(\bbx-h_{m} {{\mathbf e}_1})}{h_{m}}
-\frac{{\fff_ {1,f}}(\bbx+h_{m} {{\mathbf e}_1})- {\fff_ {1,f}}(\bbx)}{h_{m}}\leq
$$
$$(M_{n_{m}}+4)h_{m}= {\Big (} \frac{M_{n_{m}}+4}{M_{n_{m}}} {\Big )} \cdot \frac{1}{m}\leq \frac{5}{m}.$$
Since $ {\fff_ {1,f}}$ is concave this implies that $\dd_{1} {\fff_ {1,f}}(\bbx)$
exists.

Next we verify that  $\dd_{1} {\fff_ {1,f}}$ is continuous on $ {(0,1)^d}$.
We have seen that 
$\dd_{1}\fff_{1,f}
(\bbx+t{\mathbf e}_1
)$ is a monotone decreasing
continuous function in $t$ for a fixed $\bbx\in (0,1)^d$. We need to show that $\dd_{1}\fff_{1,f}
$
is continuous as a function of several variables at any $\bbx\in (0,1)^d
$.
This is quite standard. Suppose $\bbx\in(0,1)^d
$ and $\eee>0$ are fixed.
Choose $t_{0}>0$  such that 
$$\bbx\pm 2t_{0}{\mathbf e}_1
\in(0,1)^d
,\qquad
|\dd_{1}\fff_{1,f}
(\bbx\pm2t_{0}{\mathbf e}_1
)-\dd_{1}\fff_{1,f}
(\bbx)|<\frac{\eee}{2}.$$

The function $\fff_{1,f}
$ is continuous as the ``top part" of the convex hull $H_{f}$
of the continuous function $f$.
By uniform continuity of $\fff_{1,f}
$ choose $\ddd_{1}>0$  such that 
\begin{equation}\label{*A1*b}
|\fff_{1,f}
(\bbw)-\fff_{1,f}
(\bbw')|<\frac{\eee t_{0}}{4}\text{ if }
||\bbw-\bbw'||<\ddd_{1} \text{ and }
\bbw,\bbw'\in(0,1)^d
.
\end{equation}
Now suppose that 
\begin{equation}\label{*A1*a}
\text{$\bbz=(z_{1},...,z_{d})$ is a vector with $z_{1}=0$ and }||\bbz||<\ddd_{1}.
\end{equation}

Then by \eqref{*A1*b}
\begin{equation}\label{*A2*a}
\Big | \frac{\fff_{1,f}
(\bbx+\bbz+2t_{0}{\mathbf e}_1
)-\fff_{1,f}
(\bbx+\bbz+t_{0}{\mathbf e}_1
)}{t_{0}}
-
\end{equation}
$$ \frac{\fff_{1,f}
(\bbx+2t_{0}{\mathbf e}_1
)-\fff_{1,f}
(\bbx+t_{0}{\mathbf e}_1
)}{t_{0}} \Big |
<\frac{\eee}{2}.$$

By the Mean Value Theorem
  there exist $c_{\bbx}$ and $c_{\bbz}$ in $(t_{0},2t_{0})$  such that 
\begin{equation}\label{*A2*b}
 \frac{\fff_{1,f}
(\bbx+\bbz+2t_{0}{\mathbf e}_1
)-\fff_{1,f}
(\bbx+\bbz+t_{0}{\mathbf e}_1
)}{t_{0}}
=\dd_{1}\fff_{1,f}
(\bbx+\bbz+c_{\bbz}{\mathbf e}_1
)
\end{equation}
 and
$$\frac{\fff_{1,f}
(\bbx+2t_{0}{\mathbf e}_1
)-\fff_{1,f}
(\bbx+t_{0}{\mathbf e}_1
)}{t_{0}}
=\dd_{1}\fff_{1,f}
(\bbx+c_{\bbx}{\mathbf e}_1
).$$

By monotonicity of $\dd_{1}\fff_{1,f}
$ we have
\begin{equation}\label{*A2*c}
|\dd_{1}\fff_{1,f}
(\bbx+c_{\bbx}{\mathbf e}_1
)-\dd_{1}\fff_{1,f}
(\bbx)|
<|\dd_{1}\fff_{1,f}
(\bbx+2t_{0}{\mathbf e}_1
)-\dd_{1}\fff_{1,f}
(\bbx)|<\frac{\eee}{2}
.\end{equation}

From \eqref{*A2*a}, \eqref{*A2*b} and \eqref{*A2*c}
it follows that
 \begin{equation}\label{*A3*a}
 |\dd_{1}\fff_{1,f}
(\bbx)-\dd_{1}\fff_{1,f}
(\bbx+\bbz+c_{\bbz}{\mathbf e}_1
)|
 <\eee \text{ where }t_{0}<c_{\bbz}.
 \end{equation}
 
 A similar argument can show that 
 there is $c_{\bbz}'\in (t_{0},2t_{0})$
 such that
 $$ \frac{\fff_{1,f}
(\bbx+\bbz-t_{0}{\mathbf e}_1
)-\fff_{1,f}
(\bbx+\bbz-2t_{0}{\mathbf e}_1
)}{t_{0}}
=\dd_{1}\fff_{1,f}
(\bbx+\bbz-c_{\bbz}'{\mathbf e}_1
)$$
and 
 \begin{equation}\label{*A3*b}
 |\dd_{1}\fff_{1,f}
(\bbx)-\dd_{1}\fff_{1,f}
(\bbx+\bbz-c_{\bbz}'{\mathbf e}_1
)|
 <\eee \text{ where }t_{0}<c_{\bbz}'.
 \end{equation}
 
 By monotonicity of $\dd_{1}\fff_{1,f}
(\bbx+\bbz+t{\mathbf e}_1
)$
 this implies that for $-t_{0}\leq t\leq t_{0}$
 we have
 $$|\dd_{1}\fff_{1,f}
(\bbx+\bbz+t{\mathbf e}_1
)-\dd_{1}\fff_{1,f}
(\bbx)|<\eee.$$
 Since this holds for any $\bbz$ satisfying \eqref{*A1*a}
 we obtain that $\dd_{1}\fff_{1,f}
$ is continuous on $(0,1)^d
$.\end{proof}

\begin{lemma}\label{*14lem}
There exists a dense $G_{\ddd}$ set $\cag_{2,1}$ in $ {C[0,1]^d}$ such that
for every $f\in \cag_{2,1}$
\begin{equation}\label{*14a}
\dim_{H}(E_{1,f})=0,\ \dim_{H}( {\mathbf E}_{{\fff_ {1,f}}}^{\leq 1}\cap {(0,1)^d})=d-1, 
\end{equation}
$${\mathbf E}_{{\fff_ {1,f}}}^{h}\cap {(0,1)^d}=\ess \text{  for  }1<h<+\oo, \text{  and  }
\dim_{H}({\mathbf E}_{{\fff_ {1,f}}}^{+\oo})=d.$$
\end{lemma}

\begin{proof}
We choose again a countable  dense set $f_{n}\in {C[0,1]^d}$.
The functions $f_{n}$ are uniformly continuous and for $\frac{1}{16(n+m)}$
we choose $ {\eta_ {n,m}}>0$  such that 
\begin{equation}\label{*15a}
|f_{n}(\bbx)-f_{n}(\bby)|<\frac{1}{16(n+m)}\text{  if  }||\bbx-\bby||<
 {\eta_ {n,m}}, \  \bbx,\bby\in  {[0,1]^d}.
\end{equation}
We partition $ {[0,1]^d}$ into non-overlapping simplices 
$Z_{j}$, $j=1,...,\zeta_{n,m}$
 such that  the diameter of each simplex is less
than $ {\eta_ {n,m}}$. We assume that $ {V(n,m)}$ is the set of vertices of these simplices.
We can also assume that  these vertices are sufficiently independent, that is, from
 $\bbx_{1},...,\bbx_{k}\in  {V(n,m)}\cap {(0,1)^d}$, $\bbx_{i}\not=\bbx_{j}$, $i\not=j$,
 the points 
$\bbx_{j}$, $j=1,...,k$ are on the same $d-1$-dimensional hyperplane in $ {\ensuremath {\mathbb R}}^{d}$ it
follows that $k\leq d$. This means that the second assumption in Definition
\ref{*1bbdef} is satisfied.

\begin{figure}
  \begin{center}
  \includegraphics[width=14.1cm,height=8.1cm]{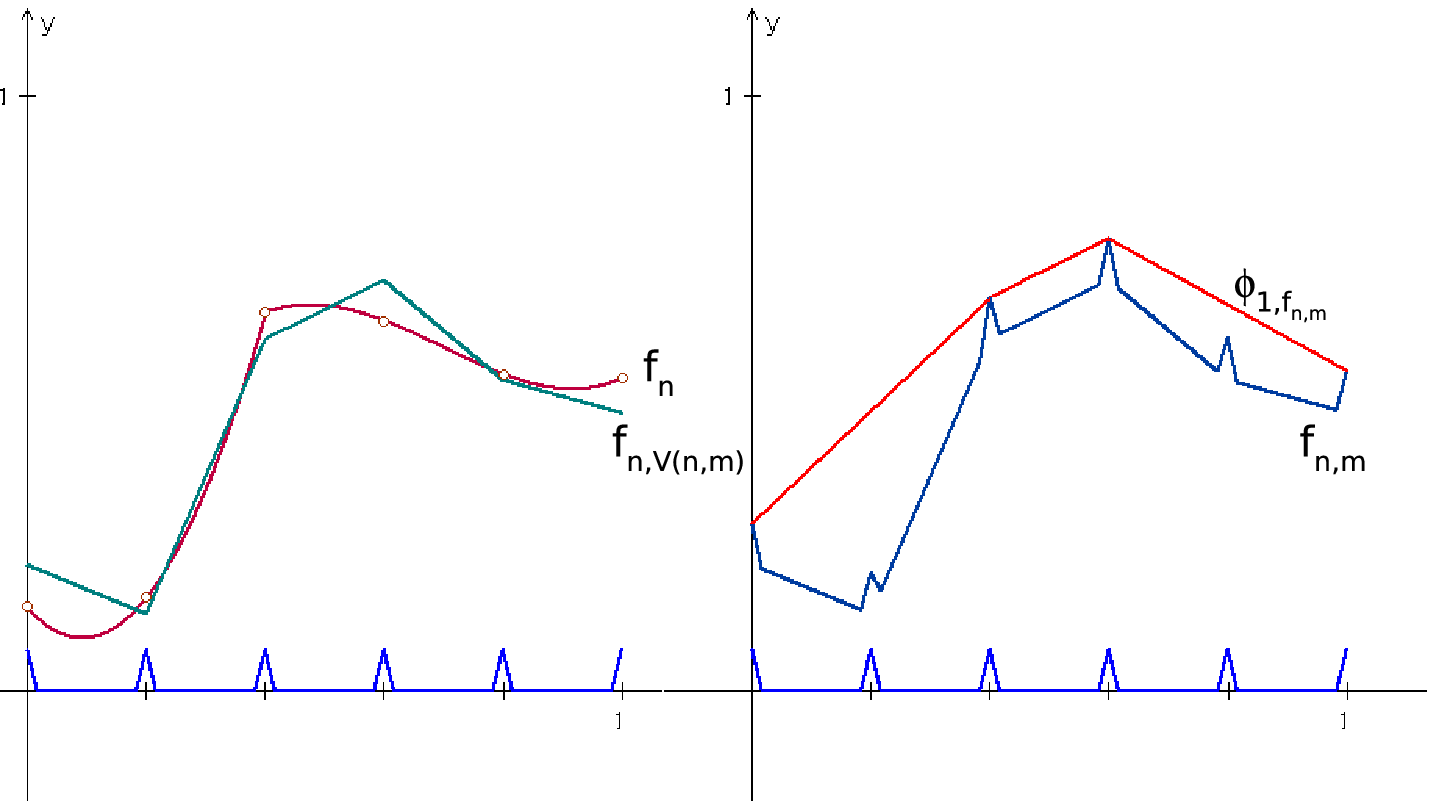}
    \caption{The functions $f_{n}$, $f_{n,V(n,m)}$, $f_{n,m}$ and $ {\fff_ {1,f_ {n,m}}}$}
     \label{fig0}
  \end{center}
\end{figure}

We denote by $ {{\widetilde {f}}_ {n,V(n,m)}}$ the function which is defined on $ {V(n,m)}$ and for any 
$\bbx\in {V(n,m)},$ $ {{\widetilde {f}}_ {n,V(n,m)}}(\bbx)=f_{n}(\bbx).$
On Figure \ref{fig0} we illustrate the procedure of selecting $f_{n,m}$. On the left half  of this figure there is $f_{n}$ with the little circles on its graph.
We suppose that on $[0,1]^1$ we used the ``simplices", which are equally spaced line
segments of length $0.2$. The function $ {{\widetilde {f}}_ {n,V(n,m)}}$ is defined on these points
and its graph is represented by the little circles on the graph of $f_{n}$.

Now we perturb $ {{\widetilde {f}}_ {n,V(n,m)}}$ a bit in order to obtain an ``independent" function $ {{\overline {f}}_ {n,V(n,m)}}$
 such that 
 \begin{equation}\label{*15b}
 | {{\overline {f}}_ {n,V(n,m)}}- {{\widetilde {f}}_ {n,V(n,m)}}|<\frac{1}{16(n+m)}
 \end{equation}
 and if $ {{\mathbf x}_0},...,\bbx_{k}\in {V(n,m)}$, $\bbx_{i}\not=\bbx_{j}$ if $i\not=j$,
 and $(\bbx_{j},{{\overline {f}}_ {n,V(n,m)}}(\bbx_{j}))$, $j=1,...,k$ are on the same hyperplane
 in $ {\ensuremath {\mathbb R}}^{d+1}$ then $k\leq d+1$.
 This implies that the first assumption of Definition \ref{*1bbdef} is satisfied
 for ${{\overline {f}}_ {n,V(n,m)}}$.

 If $\bbx\in  {[0,1]^d}\sm  {V(n,m)}$ and $\bbx$ is in the simplex $Z_{j}$,
$j\in \{ 1,...,\zeta_{n,m} \}$ with vertices
 $\bbz_{j,1},...,\bbz_{j,d+1}\in  {V(n,m)}$ we define $ {f_ {n,V(n,m)}}(\bbx)$ so that 
 $(\bbx, {f_ {n,V(n,m)}}(\bbx))$ is on the hyperplane determined by the points
 $$(\bbz_{j,1}, {{\overline {f}}_ {n,V(n,m)}}(\bbz_{j,1})),...,(\bbz_{j,d+1},
  {{\overline {f}}_ {n,V(n,m)}}(\bbz_{j,d+1})).$$
 Therefore, $ {f_ {n,V(n,m)}}$ is an independent piecewise linear function (recall Definition
 \ref{*1bbdef} and see the illustration on the left half of Figure \ref{fig0}).

 By \eqref{*15a} and \eqref{*15b}
 \begin{equation}\label{*16a}
 | {f_ {n,V(n,m)}}(\bbx)-f_{n}(\bbx)|<\frac{1}{4(n+m)}.
 \end{equation}
 
 Now we want to perturb the functions $ {f_ {n,V(n,m)}}$ a little further.
 Let $\kkk(\bbx)=\max \{1-||\bbx||,0 \}$ and for a $\ggg>0$
 put
 $$ {\Gamma_ {V(n,m),\gamma}}(\bbx)=\chi_{{[0,1]^d}}(\bbx)\sum_{\bbv\in {V(n,m)}}\kkk(\bbx/\ggg).$$
 Then $\lim_{\ggg\to 0+} {\Gamma_ {V(n,m),\gamma}}(\bbx)=\chi_{{V(n,m)}}(\bbx).$
 
 We denote by $ {\nu(n,m)}$ the minimum distance among points of $ {V(n,m)}$ and
 will select a sufficiently small $\ggg {_ {n,m}}>0$ later.
 We put
 $$ {f_ {n,m}}= {f_ {n,V(n,m)}}+\frac{1}{4(n+m)}\Gamma_{{V(n,m)},\ggg {_ {n,m}}}.$$
 On both halves of Figure \ref{fig0} one can see $\frac{1}{4(n+m)}\Gamma_{{V(n,m)},\ggg {_ {n,m}}}$ which is the function with the equally spaced small peaks
 at the points which are multiples of $0.2$. On the right half of Figure \ref{fig0}
one can see $f_{n,m}$ which is obtained from $ {f_ {n,V(n,m)}}$ (pictured on the left half of Figure \ref{fig0}) after we added the small peaks.  
 
We suppose that 
 \begin{equation}\label{*18a}
 \ggg {_ {n,m}}<\frac{{\nu(n,m)}}{100}
 \end{equation}
 and if ${\mathfrak G}(n,m)$ denotes the maximum of the norms of the gradients of the hyperplanes in the definition of $ {f_ {n,V(n,m)}}$ then
\begin{equation}\label{*18b}
\frac{1}{\ggg {_ {n,m}}}>100 \cdot {\mathfrak G}(n,m).
\end{equation} 
This way if we take the convex hull of $ {f_ {n,m}}$ then
$$ {E_ {1,f_ {n,m}}}=\{\bbx: {f_ {n,m}}(\bbx)= {\fff_ {1,f_ {n,m}}}(\bbx) \}$$
will be a subset of $ {V(n,m)}$. See the right half of Figure \ref{fig0}.
We remark that the resolution of our drawings does not make it possible to take into
all the above assumptions and hence they are distorted, but we hope that
they can help to understand our procedure.  

Our next aim is to select a sufficiently small $ {\delta_ {n,m}}>0.$
It is clear that given $r {_ {n,m}}>0$ if $ {\delta_ {n,m}}$ is sufficiently small
then 
$f$ can coincide with $ {\fff_ {1,f}}$ only close to some points
in $V(n,m)$, that is 
for any $f\in B( {f_ {n,m}}, {\delta_ {n,m}})$ and any $\bbx\in E_{1,f}$
there is $ {{\mathbf w}_ {\mathbf x}}\in {V(n,m)}$  such that 
\begin{equation}\label{*19a}
|| {{\mathbf w}_ {\mathbf x}}-\bbx||<r {_ {n,m}}.
\end{equation}
On the left half of Figure \ref{fig2} one can see $f_{n,m}$ and $f$. We also graphed
the functions $ {\fff_ {1,f_ {n,m}}}$ and $ {\fff_ {1,f}}$ which will be very close to each other.
The latter function is not exactly piecewise linear but a close approximation to such a function, namely to $ {\fff_ {1,f_ {n,m}}}$.
In the one dimensional case, like on Figure \ref{fig2} the nonlinear parts
(not pictured) are very cloes to some elements of $V(n,m)$. The higher dimensional
case is a bit more complicated and we discuss it below.

\begin{figure}
  \begin{center}
  \includegraphics[width=14.1cm,height=8.1cm]{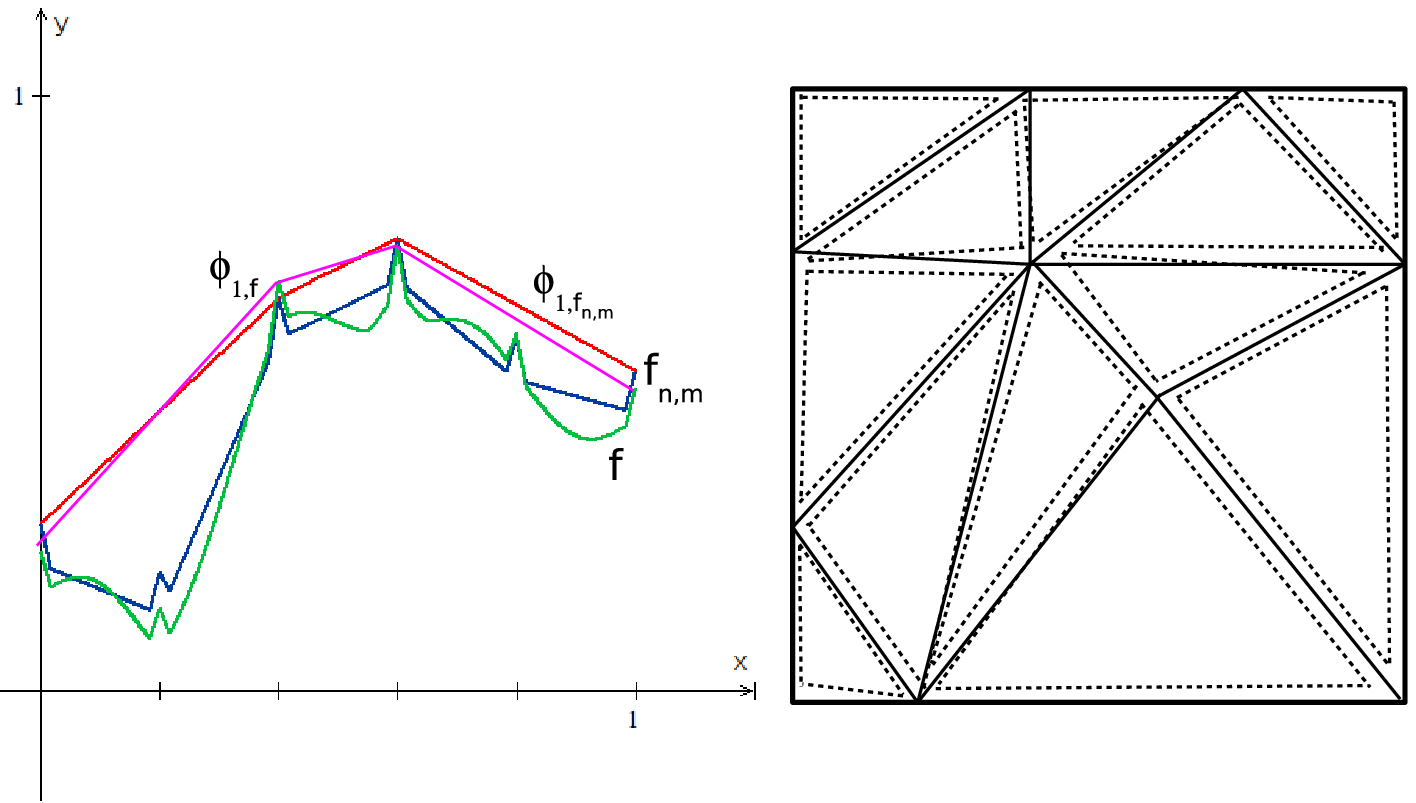}
    \caption{On the left: $f_{n,m}$ and $f$, on the right: $ {{\cal S}_ {\fff}}$ and $\Phi_{\fff,f}$ when $d=2$}
     \label{fig2}
  \end{center}
\end{figure}

We will select a sufficiently small $r {_ {n,m}}>0$ later. At this point we suppose that
\begin{equation}\label{*19b}
r {_ {n,m}}<\frac{{\nu(n,m)}}{1000}.
\end{equation}
This implies that for any $\bbx\in E_{1,f}$ there is a unique $ {{\mathbf w}_ {\mathbf x}}\in {V(n,m)}$.
That is, $f$ will ``almost" look like a piecewise linear function. This implies that if
$f\in B( {f_ {n,m}}, {\delta_ {n,m}})$ then 
\begin{equation}\label{**10a}
E_{1,f}\text{  can be covered by the set  }\cup_{\bbw\in V(n,m)}B(\bbw,r_{n,m}).
\end{equation}
We can suppose that $r {_ {n,m}}$ is chosen so small, that
\begin{equation}\label{**10b}
\#(V(n,m))\cdot r_{n,m}^{1/m}<1/m.
\end{equation}
This estimate will imply that $\dim_{H}E_{1,f}=0$ for the typical
$f\in {C[0,1]^d}$, that is for $f\in \cag_{2,1}$.

By the choice of $\ggg {_ {n,m}}$ if we consider $ {\fff_ {1,f_ {n,m}}}$ then it is an independent piecewise linear function.
There is a system of non-overlapping simplices
$$ {{\cal S}_ {\fff}}=\{S_{k}:k=1,...,s_{\fff} \}$$
 such that $(\bbx, {\fff_ {1,f_ {n,m}}}(\bbx))$ for any $\bbx\in  {[0,1]^d}$ is on a hyperplane
 determined by a simplex $S_{k}$ containing $\bbx$. On the left half of Figure \ref{fig2} the one dimensional case is illustrated and these simplices are simply the line segments
 $[0,0.4]$, $[0.4,0.6]$ and $[0.6,1]$.
 The endpoints $0.4$ and $0.6$ are points where this function "breaks"
 and these points are on two non-parallel lines ("hyperplanes").
 These breakpoints/folding regions  will be used to find those points
 where the H\"older exponent is $1$. On the right half of Figure \ref{fig2}
 the two dimensional case $d=2$ is illustrated. This time we have simplices (triangles) in $[0,1]^{d}$
 bounded by solid lines on which  $ {\fff_ {1,f_ {n,m}}}$ is linear. 
 On the right half of Figure \ref{fig2} only the domain of $ {\fff_ {1,f_ {n,m}}}$
 is shown.
 The system of the
 simplices (triangles) bounded with  dashed lines will be simplices corresponding to $ {\fff_ {1,f}}$. Later we will explain this in more detail.
 
 By the independence property of $ {f_ {n,V(n,m)}}$ the hyperplanes
determined by the simplices $S_{k}$
 are different
 for different $S_{k}$.

 We denote by $V_{\fff}$ the set of the vertices of the simplices $S_{k}$,
 $k=1,...,s_{\fff}.$ Clearly, $V_{\fff}\sse  {V(n,m)}.$
 The union of the faces of these simplices will be denoted by $\Phi_{\fff}=
 \cup_{k=1}^{s_{\fff}}\dd(S_{k}).$
 If $\bbx_{0}\in \dd(S_{k})\cap \dd(S_{k'})$ with $k\not=k'$
 then $\{(\bbx, {\fff_ {1,f_ {n,m}}}(\bbx)):\bbx\in S_{k} \}$
 and $\{(\bbx, {\fff_ {1,f_ {n,m}}}(\bbx)):\bbx\in S_{k'} \}$
 are on different hyperplanes and hence the graph of
 $ {\fff_ {1,f_ {n,m}}}$ ``breaks" at $\bbx_{0}$. 
 This implies that we can choose $ {\tau_ {1,n,m}}$  such that for any $ {{\mathbf x}_\fff}\in\Phi_{\fff}\cap  {(0,1)^d}$ and for any hyperplane $L_{\bbx}$ passing through
 $( {{\mathbf x}_\fff}, {\fff_ {1,f_ {n,m}}}( {{\mathbf x}_\fff}))$ one can choose a point $ {\mathbf x}_\fff'$
  such that 
\begin{equation}\label{*22a}
 {\tau_ {1,n,m}}\leq \dist( {\mathbf x}_\fff',\Phi_{\fff})\leq \frac{1}{n+m},
\end{equation}
and
\begin{equation}\label{*22b}
|L_{\bbx}( {\mathbf x}_\fff')- {\fff_ {1,f_ {n,m}}}( {\mathbf x}_\fff')|\geq || {{\mathbf x}_\fff}- {\mathbf x}_\fff'||^{1+\frac{1}{m}}\geq
 {\tau_ {1,n,m}}^{1+\frac{1}{m}}, 
\end{equation}
where we used that \eqref{*22a} implies $|| {\mathbf x}_\fff'- {{\mathbf x}_\fff}||> {\tau_ {1,n,m}}$.
See Figure \ref{fig4}.

\begin{figure}
  \begin{center}
  \includegraphics[width=14.1cm,height=8.1cm]{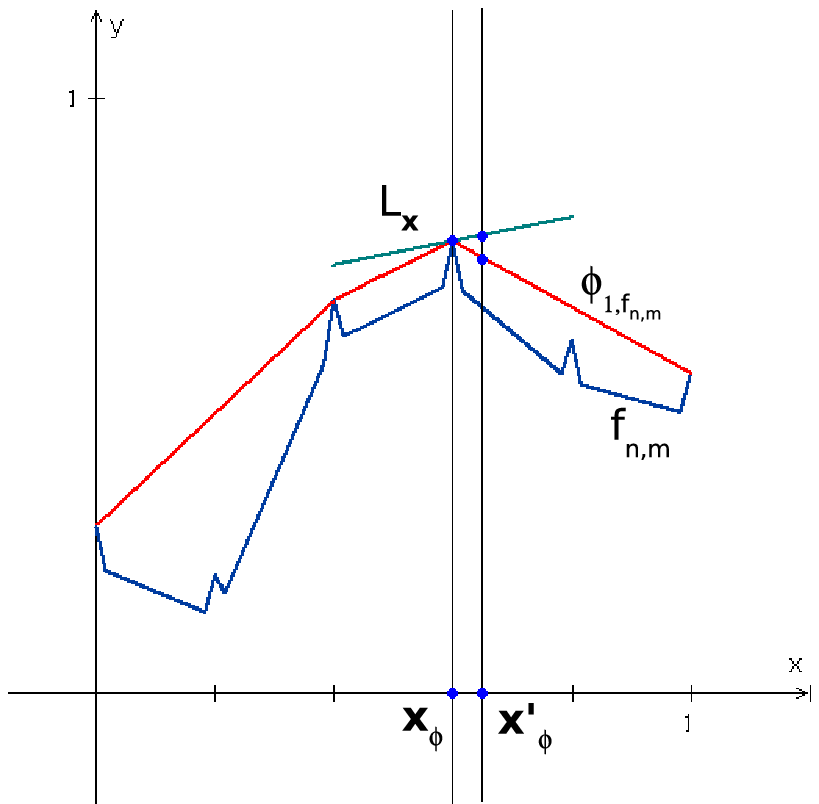}
    \caption{The breaking point at  $( {{\mathbf x}_\fff}, {\fff_ {1,f_ {n,m}}}( {{\mathbf x}_\fff}))$}
     \label{fig4}
  \end{center}
\end{figure}

It is also clear that if $f$ is a good approximation of $ {f_ {n,m}}$ then one can see
similar ``breaking" properties on $\fff_{1,f}$. This time there are no ``folding edges"
like in the case of $ {f_ {n,m}}$ on $\Phi_{\fff}\cap  {(0,1)^d}$ but there are
regions around $\Phi_{\fff}$ where we can see similar phenomena.

Using that $ {f_ {n,m}}$ and $ {\fff_ {1,f_ {n,m}}}$ are both independent piecewise linear
functions one can see that
\begin{equation}\label{*23a}
 {\Delta_ {1,n,m}}(\ddd)=\sup\{| {\fff_ {1,f_ {n,m}}}(\bbx)- {\fff_ {1,f}}(\bbx)|:\bbx\in {[0,1]^d},\ f\in B( {f_ {n,m}},\ddd) \}\to 0 
\end{equation}
$\text{  as  }\ddd\to 0+.$
Apart from \eqref{*19b} and \eqref{**10b} we also assume that $r {_ {n,m}}>0$ is chosen so small 
that
\begin{equation}\label{*23b}
r {_ {n,m}}<\frac{{\tau_ {1,n,m}}}{100}
\end{equation}
and (using that $\dim_{H}\Phi_{\fff}=d-1$)
\begin{equation}\label{*23c}
B(\Phi_{\fff},r {_ {n,m}})\text{  can be covered by balls  }B_{i}\text{  such that  }
\end{equation}
$$|B_{i}|<\frac{1}{n+m}$$
and
\begin{equation}\label{*23d}
\sum_{i}|B_{i}|^{(d-1)+\frac{1}{m}}<\frac{1}{m}.
\end{equation}

Recall that we started to make assumptions about $ {\delta_ {n,m}}$ in the paragraph
containing \eqref{*19a}. The smaller $r {_ {n,m}}$ we need to use the smaller
$ {\delta_ {n,m}}$. Next we suppose that using \eqref{*23a}
we chose a $ {\delta_ {n,m}}$  such that
in addition to our other assumptions we have  
\begin{equation}\label{*24a}
 {\Delta_ {1,n,m}}( {\delta_ {n,m}})<\frac{1}{100}  {\tau_ {1,n,m}}^{1+\frac{1}{m}}.
\end{equation}

Now we want to use the folding property in \eqref{*22a} and \eqref{*22b}
for functions $f$ which approximate $ {f_ {n,m}}$.
This time the ``folding edges" are not any more $(d-1)$-dimensional
surfaces, but some neighborhoods of them.
By \eqref{*23c} and \eqref{*23d} we will be able to bound the dimension
of these regions.

Suppose $S_{k}\in  {{\cal S}_ {\fff}}$ with $k\in \{1,...,s_{\fff} \}$
with vertices $\bbz_{k,1},...,\bbz_{k,d+1}$.
Since $ {f_ {n,m}}$ and $ {\fff_ {1,f_ {n,m}}}$ are both 
independent piecewise linear functions 
there is $\ddd_{\fff,k}>0$  such that if $f\in B( {f_ {n,m}},\ddd_{\fff,k})$
then one can choose vertices $\bbz_{k,j,f}$, $j=1,...,d+1$
 such that 
 \begin{equation}\label{*25a}
 ||\bbz_{k,j}-\bbz_{k,j,f}||<r {_ {n,m}} \text{  and  }\bbz_{k,j,f}\in E_{1,f},\ j=1,...,d+1,
 \end{equation}
moreover if $S_{k,f}$ denotes the simplex determined by
$\{\bbz_{k,j,f}: j=1,...,d+1 \}$ then $\{(\bbx,f(\bbx)):\bbx\in S_{k,f} \}$ is on the surface of $H_{f}$ inside a hyperplane determined
by
$\ds \{(\bbz_{k,j,f},f(\bbz_{k,j,f})):k=1,...,d+1 \}$, that is,
$\{(\bbx,{\fff_ {1,f}}(\bbx)):\bbx\in S_{k,f} \}$ is a ``face" of $ {\fff_ {1,f}}$
approximating $\{(\bbx,\fff_{1,f_{n,m}}(\bbx)):\bbx\in S_{k} \}$.
On the right half of Figure \ref{fig2} we have the two dimensional illustration.
The simplices (triangles) $S_{k}\in\cas_{\fff}$ are bounded by solid lines.
The simplices (triangles) $S_{k,f}$ are bounded by dashed lines. 

We can suppose that $ {\delta_ {n,m}}<\min\{\ddd_{\fff,k}:k=1,...,s_{\fff} \}$
and by using independent piecewise linearity of $ {f_ {n,m}}$
and $ {\fff_ {1,f_ {n,m}}}$ we obtain that the hyperplanes containing
$\{(\bbx,{\fff_ {1,f}}(\bbx)):\bbx\in S_{k,f} \}$
are different for different $k$'s.

Hence the simplices $S_{k,f}$ are non-overlapping. 

Put
\begin{equation}\label{*26a}
\Phi_{\fff,f}= {(0,1)^d}\sm \cup_{k=1}^{s_{\fff}} \intt(S_{k,f}).
\end{equation}

These sets $\Phi_{\fff,f}$
will replace the folding edges $\Phi_{\fff}\cap  {(0,1)^d}$.
On Figure \ref{fig2} this is the region which is not covered by
the interiors of the simplices (triangles) bounded by dashed lines.

From $||\bbz_{k,j}-\bbz_{k,j,f}||<r {_ {n,m}}$ in \eqref{*25a} it follows that any point $\bbx$ in $S_{k}$ which is of distance no less than $r {_ {n,m}}$
from $\dd(S_{k})$
is covered by $S_{k,f}$.

Thus $\Phi_{\fff,f}\sse B(\Phi_{\fff},r {_ {n,m}})$ and hence
by \eqref{*23c} and \eqref{*23d}
\begin{equation}\label{*27a}
\Phi_{\fff,f} \text{  can be covered by balls  }B_{i} \text{  such that  }
\end{equation}
$$|B_{i}|<\frac{1}{n+m}\text{  and  }\sum_{i}|B_{i}|^{d-1+\frac{1}{m}}<\frac{1}{m}.$$

Using all the above restrictions we can select $ {\delta_ {n,m}}>0$.

Set $\cag_{2,1}=\cap_{m=1}^{\oo}\cup_{n=1}^{\oo}B( {f_ {n,m}}, {\delta_ {n,m}})$.

It is clear that $\cag_{2,1}$ is a dense $G_{\ddd}$ set in $ {C[0,1]^d}$.

Suppose $f\in\cag_{2,1}$. Then there exists a sequence $n_{m}$  such that 
$f\in B( {{f_ {n_ {m} ,m}}}, {\delta_ {n_m,m}})$. For each $m$ we can define the "folding region" as
in \eqref{*26a}. Since these regions depend on $m$ we denote them by
$\Phi_{\fff,f,m}$.
Set $\Phi_{f}=\cap_{m=1}^{\oo}\Phi_{\fff,f,m}$.
If $\bbx\in  {(0,1)^d}\sm\Phi_{f}$ then there exists an $m$  such that 
$\bbx\in\intt(S_{k,f,m})$ with a simplex
$S_{k,f,m}$
 and $\{(\bbx,\fff_{1,f}(\bbx)):\bbx\in S_{k,f,m} \}$ is the subset of a hyperplane in $ {\ensuremath {\mathbb R}}^{d+1}$. This implies that $ {\fff_ {1,f}}$
is locally linear in a neighborhood of $\bbx$ and $h_{ {\fff_ {1,f}}}(\bbx)=+\oo$.

Using \eqref{*27a} one can easily see that $\dim_{H}(\Phi_{f})\leq d-1.$
On the other hand, from \eqref{*27a} it also follows that if $S\sse {(0,1)^d}$ 
is a simplex such that its vertices are $\bbz_{1},...,\bbz_{d+1}\in E_{1,f}$
then  there exists $m_{0}$  such that for $m\geq m_{0}$, $S\not\sse \Phi_{\fff,f,m}$. Since $S$ is a ``face" of $ {\fff_ {1,f}}$ if $\bbx\in\dd (S)$ then $\bbx$ cannot belong to the interior of any other ``face" of
$ {\fff_ {1,f}}$. Hence $\dd(S)\sse \Phi_{f}.$ Since $\dim_{H}\dd(S)=d-1$
we obtain  that $\dim_{H}(\Phi_{f})=d-1$.
If $\bbx\in \Phi_{f}$
then \eqref{*22a} and \eqref{*22b} imply that $h_{f}(\bbx)\leq 1.$

The property $\dim_{H}E_{1,f}=0$ follows from \eqref{**10a} and \eqref{**10b}.

\end{proof}

\begin{proof}[Proof of Theorem \ref{*30th}]
For $j=1,...,d$ a version of Lemma \ref{*2lem} can provide
dense $G_{\ddd}$ sets $\cag_{0,j}\sse {C[0,1]^d}$
 such that \eqref{*2a} and \eqref{*2b} hold with $d$ replaced by $j$.
 If we use the faces $F_{1,j}$ instead of $F_{0,j}$ in analogous versions of
 Lemma \ref{*2lem} we can obtain dense $G_{\ddd}$ sets
 $\cag_{1,j}\sse {C[0,1]^d}$.
 
Taking $\cag_{0}=\cap_{j=1}^{d}\cag_{0,j}\cap \cag_{1,j}$ for any $f\in \cag_{0}$
we have \eqref{*30c}, \eqref{*30b}
and \eqref{*30a} satisfied.

By Lemma \ref{*6lem}  there exists a dense open set $\cag_{1}\sse  {C[0,1]^d}$
 such that for any $f\in\cag_{1}$ the functions $ {\fff_ {1,f}}$ and $ {\fff_ {2,f}}$
are continuously  differentiable on $ {(0,1)^d}$.

There is nothing special about $ {\fff_ {1,f}}$ in Lemma \ref{*14lem}. A similar lemma can provide a dense $G_{\ddd}$ set $\cag_{2,2}$
 such that for any $f\in\cag_{2,2}$ we have 
 \eqref{*14a}
for $E_{2,f}$ and $ {\fff_ {2,f}}$.

If we take $\cag=\cag_{0}\cap\cag_{1}\cap\cag_{2,1}\cap\cag_{2,2}$
then 
taking into consideration
 Remark \ref{*rem7}
 as well
any $f\in\cag$ satisfies the conclusions of Theorem \ref{*30th}.
 \end{proof}

\end{document}